\documentclass[a4paper,11pt]{article}
\pdfoutput=1 

\usepackage{url}
\usepackage{amsmath}
\usepackage{amssymb}
\usepackage{amsfonts}
\usepackage{amssymb}
\usepackage{amsthm}
\usepackage{amscd}
\usepackage{epsfig}
\usepackage{epsf}
\usepackage{graphicx}
\usepackage{float}
\usepackage[latin1]{inputenc}
\usepackage[english]{babel}
\usepackage{multicol}
\usepackage{subfig}

\graphicspath{{Figures/}}

\def \mod {\mathop{\rm mod}\nolimits}

\def\eref#1{(\ref{#1})}

\def \N {\mathbb{N}}
\def \P {\mathbb{P}}
\def \A {\mathcal{A}}

\def \C {\mathcal{C}}
\def \D {\mathcal{D}}
\def \M {\mathcal{M}}

\def \Z {\mathbb{Z}}

\def \M {\mathcal{M}}
\def \A {\mathcal{A}}
\def \C {\mathcal{C}}
\def \Neighb {\mathcal{N}}
\def \ber {{\mathcal{B}}}
\def \N {\mathbb N}

\def \Z {\mathbb Z}
\def \P {{\mathbb P}}
\def \X {{\mathcal X}}
\def \a {\alpha}

\newtheorem{defi}{Definition}[section]
\newtheorem{prop}[defi]{Proposition}
\newtheorem{lemm}[defi]{Lemma}
\newtheorem{cor}[defi]{Corollary}
\newtheorem{theo}[defi]{Theorem}

\newtheorem{exam}[defi]{Example}
 
\begin{document}

\title{{\bf Probabilistic cellular automata \\ and random fields with i.i.d. directions}\thanks{This work was partially supported by the ANR project
 MAGNUM (ANR-2010-BLAN-0204).}}

\author{
Jean {\sc Mairesse}%
\thanks{CNRS, UMR 7089, LIAFA, Univ Paris Diderot, Sorbonne
  Paris Cit\'e, F-75205 Paris, France. 
E-mail: {\tt Jean.Mairesse@liafa.univ-paris-diderot.fr}.}
\and 
Ir\`ene {\sc Marcovici}%
\thanks{Univ Paris Diderot, Sorbonne Paris Cit\'e, LIAFA, UMR 7089
  CNRS, F-75205 Paris, France. 
E-mail: {\tt Irene.Marcovici@liafa.univ-paris-diderot.fr}. }}

\date{\today}

\maketitle

\begin{abstract}
Let us consider the simplest model of one-dimensional probabilistic
cellular automata (PCA). The cells are indexed by the integers, the
alphabet is $\{0,1\}$, and all the cells evolve synchronously. The new content
of a cell is randomly chosen, independently of the others, according to a
distribution depending only on the content of the cell itself and of its
right neighbor.
There are necessary and sufficient conditions on the four parameters
of such a PCA to have a Bernoulli product invariant measure. We study the
properties of the random field given by the space-time diagram obtained
when iterating the PCA starting from its Bernoulli product invariant measure. 
It is a non-trivial random 
field with very weak dependences and nice combinatorial properties. In
particular, not only the horizontal lines but also the lines in any
other direction consist in i.i.d. random
variables. We study extensions of the results to Markovian invariant measures, 
and to PCA with larger alphabets and neighborhoods.
\end{abstract}

\tableofcontents

\section{Introduction}

Consider a bi-infinite set of cells indexed by the integers $\Z$, each cell
containing a letter from a finite alphabet $\A$. The updating 
is local (each cell updates according to a finite neighborhood), time-synchronous, and space-homogeneous. When the updating is
deterministic, we obtain a Cellular Automaton (CA), and when it is
random, we obtain a Probabilistic Cellular Automaton (PCA). 
Alternatively, a PCA may be viewed as the discrete-time and synchronous counterpart of 
a (finite range) interacting particle system. We refer to
\cite{toombook} for a comprehensive survey of the theory of PCA. 

\medskip


There are two complementary viewpoints on PCA. First, it defines a
mapping from the set of probability measures on $\A^{\Z}$ into itself. 
Second, it defines a discrete-time Markov chain on the state space
$\A^{\Z}$. A realization of the Markov chain defines a random field on
$\A^{\Z \times \N}$, called a {\em space-time diagram}. An {\em
  invariant measure} for a PCA is a probability measure on $\A^{\Z}$ which is
left invariant by the dynamic. Starting from an invariant measure, we obtain
a space-time diagram which is time-stationary.  
Our goal is to study the stationary random fields associated to some particular and remarkable 
PCA. 

\medskip

First, we consider the image by a PCA of a Bernoulli product measure. The resulting measure is described via 
explicit formulas for its finite-dimensional marginals. 
Second, we use this description to revisit a result from \cite{bel}
(see also \cite{vas,toombook}) with a new and simple proof: 
explicit conditions on a PCA
ensuring that a Bernoulli product measure is invariant. Third, we focus on the equilibrium behavior of PCA having such a Bernoulli product invariant measure. 
The resulting space-time diagram turns out to have an original and subtle
correlation structure: it is non-i.i.d. but, in any direction, the
``lines'' are i.i.d. 
In the case of an alphabet of size two and a neighborhood of size two (the
updating of a cell depends only on itself and its
right-neighbor),
the stationary space-time diagram satisfies additional remarkable
properties: it can also be seen as being  obtained by iterating a transversal PCA in
another direction. 

\medskip

The paper is structured as follows. General definitions are given in Section 2. A special emphasis is put on the simplest non-trivial PCA, that is, the ones 
defined on an alphabet of size 2 and a neighborhood of size 2. They are studied in details in Section 3 and 4. In
Section 5, we consider the extension to general alphabets and
neighborhoods, and we also consider the case of Markovian invariant
measures.  
In Section 6, we revisit classical results on CA in view of the PCA
results.

\medskip

{\bf Notations.} Given a finite set $\A$, the free
semigroup generated by $\A$ is denoted by  $\A^+$. 
The length, that is, number of letters, of a word $u\in \A^+$ is
denoted by $|u|$. The number of occurences of the letter $a\in \A$ in
a word $u\in \A^+$ is denoted by $|u|_a$. 

\section{Probabilistic cellular automata (PCA)}

\subsection{Definition of PCA}

Let $\A$ be a finite set, called the \emph{alphabet}, and let ${\X}=\A^{\Z}$. The set $\Z$ will be referred to as the set of \emph{cells}, whereas $\X$ is the set of \emph{configurations}.
For some finite subset $K$ of $\Z$, consider $y=(y_k)_{k\in K}\in \A^K$. The
\emph{cylinder} defined by $y$ is the set
\begin{equation*}
[y] = \bigl\{x\in {\X} \mid
\forall k\in K, x_k=y_k \bigr\} \:.
\end{equation*}
For a given finite subset $K$, we denote by $\C(K)$
the set of all cylinders of base $K$.  Given $K,L \subset \Z$, we
define $K+L = \{ u+ v \mid u\in K, v\in L\}$. 

We denote by
$\M(\A)$ the set of probability measures on $\A$. 
Let us equip ${\X}$ with the product topology, which can be
described as the topology generated by cylinders. We denote by
 $\M({\X})$ the set of probability measures on ${\X}$ for the
 Borelian  $\sigma$-algebra.


\begin{defi}\label{de-pca}
Given a finite set $\Neighb \subset \Z$, a \emph{transition function} of \emph{neighborhood} $\Neighb$
is a function $f: \A^{\Neighb} \rightarrow  \M(\A)$. 
The \emph{probabilistic cellular automaton} (PCA) of transition
function $f$ is the application $F: \M(\X)  \longrightarrow \M(\X), \ \mu
\longmapsto \mu F$, defined on cylinders by: $\forall K, \forall y=(y_k)_{k\in
  K}$, 
\[
\mu F[y]=\sum_{[x]\in \C(K+\Neighb)}\mu[x]\prod_{k\in K}f((x_{k+v})_{v\in \Neighb})(y_k) \:.
\]
\end{defi}

Assume that the initial measure is concentrated on some configuration $x\in \X$. Then by application of
$F$, the content of the $k$-th cell is updated to $a\in\A$ with probability 
$f((x_{k+v})_{v\in \Neighb})(a)$. 

\medskip

We keep the notation $f$ for the extended mapping
$\M(\A^{\Neighb}) \longrightarrow \M(\A), \ \nu \longmapsto \nu f$
with 
\[
\forall a \in \A, \qquad \nu f(a)= \sum_{u\in \A^{\Neighb}}\nu(u)
f(u)(a) \:.
\]

\subsection{Space-time diagrams}

A PCA is a Markov chain on the state space ${\X}$. 
Consider a realization $(X^n)_{n\in \N}$ of that Markov chain. If
$X^0$ is distributed according to $\mu$ on $\X$, then $X^n$ is
distributed according to $\mu F^n$. The random field $(X^n)_{n\in \N}
= (X^n_k)_{k\in \Z,n\in \N}$ is called a \emph{space-time diagram}
(the space-coordinate is $k$, and the time-coordinate is $n$). 

If the neighborhood is $\Neighb=\{0,1\}$, for symmetry reasons, a natural
choice is to represent the space-time diagram on a regular triangular
lattice, as in Figure \ref{fi-spacetime}. 




\begin{figure}
\begin{center}
\includegraphics[scale=0.60]{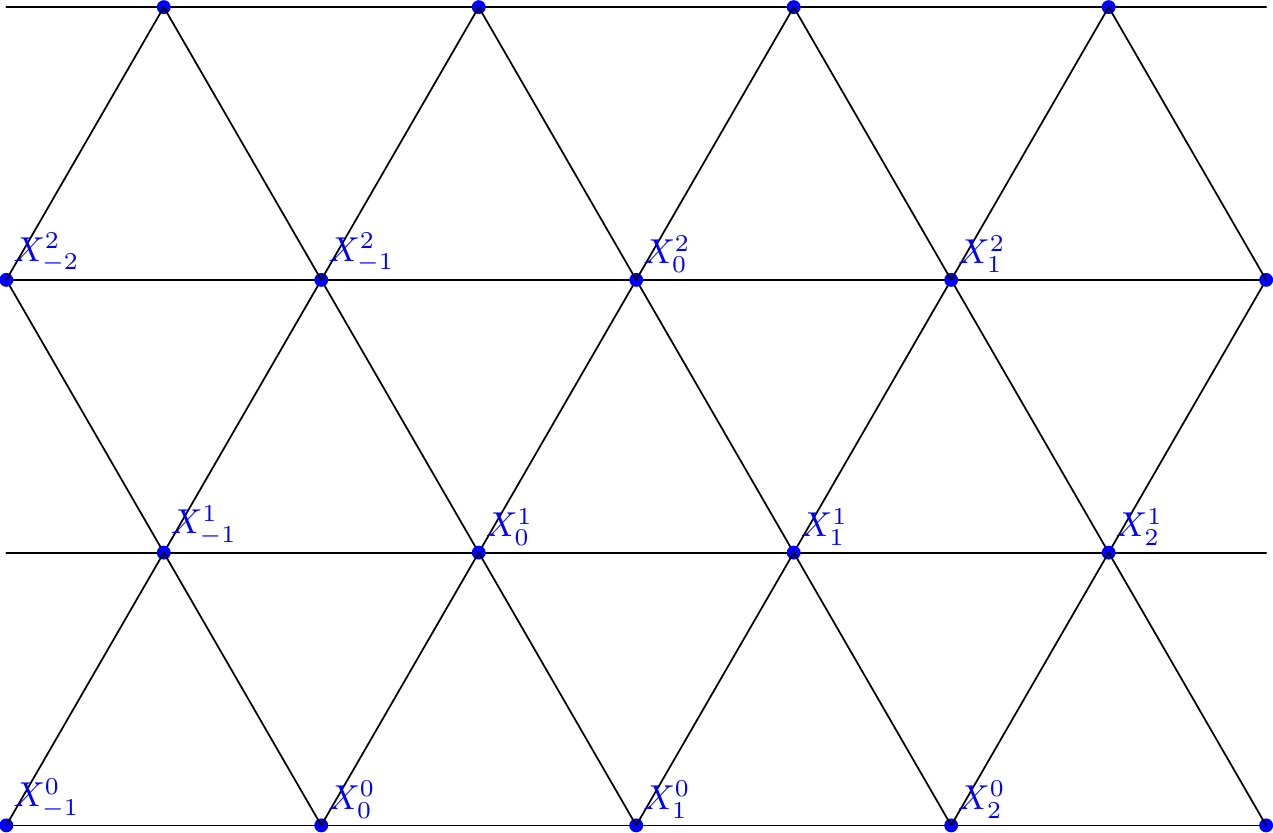}
\caption{Space-time diagram.}\label{fi-spacetime}
\end{center}
\end{figure}

The \emph{dependence cone} $\D(i,n)$ of the variable $X^n_i$ is defined as
the set
of variables which are influenced by the value of $X^n_i$. If the
neighborhood is $\Neighb=\{0,\dots \ell\}$, then
$\D(i,n)=\{X^{n+k}_{i+j}, k\in\N, -k\ell \leq j\leq 0 \}.$ 

Next lemma follows directly from the definition of a PCA. 

\begin{lemm}\label{lem:cone} Let $(i,n)$ belong to $\Z\times (\N\setminus\{0\})$
  and let $S$ be a subset of $\Z\times\N$ such that $\D(i,n)\cap
  S=\emptyset$. Then, $X^n_i$ is independent of $(X^m_j)_{(j,m)\in S}$
  conditionally to $(X^{n-1}_{i+v})_{v\in\Neighb}$. 
\end{lemm}

We point out that if a PCA has \emph{positive rates},
i.e., $\forall u\in \A^{\Neighb}, \forall a \in \A, \ f(u)(a)>0$,
then any of its stationary space-time diagram is a Markovian random field. We refer
to \cite{leb} for an in-depth study of the connections between Gibbs
states and stationary space-time diagrams of PCA. 

\subsection{Product form invariant measures}

\begin{defi} For $p\in[0,1]$, we denote by $\mu_p$ the Bernoulli
  product measure of parameter $p$ on $\{0,1\}^{\Z}$, that is,
  $\mu_p=\ber_{p}^{\otimes \Z}$, where $\ber_{p}$ denotes the Bernoulli
  measure of parameter $p$ on $\{0,1\}$. Thus, for any cylinder $[x]$, we have 
$$\mu_p[x]=(1-p)^{|x|_0}p^{|x|_1}\:.$$ 
\end{defi}


We give a first property of the space-time diagram that is
shared by every PCA having a Bernoulli product invariant measure. 

\begin{lemm}\label{outside} Let $F$ be a PCA of neighborhood
  $\{0,\ldots,\ell\}$. Assume that $\mu_pF=\mu_p$ and consider
  the stationary space-time diagram obtained for that invariant
  measure. 
 Then for any $\alpha>-1/\ell$, the line $L_{\alpha}=\{(k,n)\in\Z\times\N \mid n=\alpha k\}$ is 
 such that the random variables $(X_k^n)_{(k,n)\in L_{\alpha}}$ are i.i.d.
\end{lemm}

\begin{proof} Let us show that any finite sequence of 
consecutive random variables on such a line is i.i.d. We can assume without loss
of generality that the first of these points is $X_0^0$. Then,
using the hypothesis on the slope, we obtain that the other random 
variables on that line are all outside the dependence cone of $X_0^0$. Thus, the
$(n-1)$-tuple they constitute is independent of $X_0^0$. By induction,
we get the result. \end{proof}

\section{PCA of alphabet and neighborhood of size 2}\label{cond}

For the time being, we assume that the  neighborhood is
$\Neighb=\{0,1\}$ and that the alphabet is $\A=\{0,1\}$. 
For convenience, we introduce the
notations: for $x,y \in \A$, 
\begin{equation*}
\theta_{xy}=\theta^1_{xy}=f(x,y)(1),
\qquad 
\theta^0_{xy}=f(x,y)(0)=1-\theta_{xy}\:.
\end{equation*}
Observe that a PCA is completely characterized by the four parameters: $\theta_{00}, \theta_{01},
\theta_{10}$, and $\theta_{11}$. 



\subsection{Computation of the image of a product measure by a
  PCA}\label{computation}

The goal of this section is to give an explicit description of the
measure $\mu_pF$, where $\mu_p$ is the Bernoulli product measure of
parameter $p$,  as a function of the parameters $\theta_{00}, \theta_{01},
\theta_{10},\theta_{11}$. 

\medskip

Let us start with an observation. Consider $(Y_n)_{n\in \Z}\sim \mu_p F$. Clearly, we have: $(Y_{2n})_{n\in \Z}\sim \mu_p$ and 
$(Y_{2n+1})_{n\in \Z}\sim \mu_p$. But the two i.i.d. sequences have a complex joint correlation structure. It makes it non-elementary
to describe the finite-dimensional marginals of $\mu_p F$. 

\medskip


\medskip


Assume that the parameters satisfy:
\begin{equation}\label{eq-condition}
(\theta_{00},\theta_{01}), (\theta_{10},\theta_{11}) \not\in
  \{(0,0), (1,1)\}\:.
\end{equation}

For $p\in (0,1)$, $\alpha\in\{0,1\}$, define the function
\begin{equation}\label{eq-g}
\begin{array}{cccl}
g_{\alpha}: & [0,1] & \longrightarrow &(0,1) \\
&  q & \longmapsto & (1-q)\;(1-p)\;\theta_{00}^{\alpha}+(1-q)\;
p\;\theta_{01}^{\alpha}+q\;
(1-p)\;\theta_{10}^{\alpha}+q\; p\;\theta_{11}^{\alpha}\:.
\end{array}
\end{equation}

Consider three random variables $X_0,X_1,Y_0$ with $(X_0,X_1) \sim
\ber_q\otimes \ber_p$ and $Y_0 \sim (\ber_q\otimes \ber_p) f$.
In words, $g_{\alpha}(q)$ is the probability to have $Y_0 =\alpha$.
With the condition \eref{eq-condition}, we have $g_{\alpha}(q)\in
(0,1)$ for all $q$. Observe also that: $g_{0}(q)+g_1(q)=1$ for all
$q$. 

For $p\in (0,1)$,
$\alpha\in\{0,1\}$,  we also define the function
\begin{equation}\label{eq-h}
\begin{array}{cccl}
h_{\alpha}: & [0,1] & \longrightarrow & [0,1]\\
 & q & \longmapsto & \bigl[ (1-q)\; p \; \theta_{01}^{\alpha}+ q\; p\;
  \theta_{11}^{\alpha}\bigr] g_{\alpha}(q)^{-1}\:.
\end{array}
\end{equation}
Consider $X_0,X_1,Y_0$ with $(X_0,X_1) \sim
\ber_q\otimes \ber_p$ and $Y_0 \sim (\ber_q\otimes \ber_p) f$. In words, $h_{\alpha}(q)$ is the probability to have $X_1 =1$
conditionally to $Y_0=\alpha$.



\begin{prop}\label{formule_vois2} 
Consider a PCA satisfying \eref{eq-condition}. Consider $p\in (0,1)$.  For $\alpha_0\cdots \alpha_{n-1} \in \A^n$, the probability of the
cylinder $[\alpha_0\cdots \alpha_{n-1}]$ under $\mu_p F$ is given by:
$$\mu_pF[\alpha_0\cdots\alpha_{n-1}]= g_{\alpha_0}(p) \ \prod_{i=1}^{n-1}
g_{\alpha_i}(h_{\alpha_{i-1}}(h_{\alpha_{i-2}}(\ldots
h_{\alpha_0}(p)\ldots )))\:.$$ 
\end{prop}

By reversing the space-direction, we get an analog proposition for a PCA satisfying the symmetrized condition: $(\theta_{00},\theta_{10}), (\theta_{01},\theta_{11}) \not\in
  \{(0,0), (1,1)\}$. 

\begin{proof} Let us compute recursively the value
  $\mu_pF[\alpha_0\cdots\alpha_{n-1}]$. We set $X=X^0$ and $Y=X^1$. Assuming that $X\sim \mu_p$, by definition,
$$\mu_pF[\alpha_0]=\P(Y_0=\alpha_0)=g_{\alpha_0}(p).$$ 
We can decompose the probability $\mu_pF[\alpha_0\alpha_1]$ into
$$\mu_pF[\alpha_0\alpha_1]=\P(Y_0=\alpha_0, Y_1=\alpha_1)=\P(Y_1=\alpha_1\mid Y_0=\alpha_0)\; \P(Y_0=\alpha_0).$$
By definition, the conditional law of $X_1$ assuming that $Y_0=\alpha_0$ is
given by $\ber_{h_{\alpha_0}(p)}$. So the law of $(X_1,X_2)$ is
$\ber_{h_{\alpha_0}(p)}\otimes\ber_{p}$ and we obtain 
$$\mu_pF[\alpha_0\alpha_1]=g_{\alpha_1}(h_{\alpha_0}(p))\; g_{\alpha_{0}}(p).$$

More generally, we have: 
$$\P(Y_0=\a_0\ldots Y_k=\a_k)=\P(Y_k=\a_k\mid Y_0=\a_0\ldots Y_{k-1}=\a_{k-1})\; \P(Y_0=\a_0\ldots Y_{k-1}=\a_{k-1}).$$
By induction, the law of $X_k$ knowing that $Y_0=\a_0\ldots Y_{k-1}=\a_{k-1}$ is $\ber_{h_{\alpha_{k-1}}(h_{\alpha_{k-2}}(\ldots h_{\alpha_0}(p)\ldots
))}$. The result follows. \end{proof}

\subsection{Conditions for a product measure to be invariant}

For $x\in \X$, denote by $\delta_x$ the Dirac probability measure
concentrated on the configuration $x$.  
The probability measure
$\mu_1= \delta_{1^{\Z}}$ is invariant for the PCA $F$ if and only if $\theta_{11}=1$.
Similarly, $\mu_0= \delta_{0^{\Z}}$ is invariant for $F$ if and only if
$\theta_{00}=0$. 

Using Proposition \ref{formule_vois2}, we get a necessary and sufficient 
condition for $\mu_p$, $p\in (0,1)$, to be an invariant measure of
$F$. The result is stated in Theorem \ref{iff}. It already appeared in 
\cite{bel} and \cite{toombook}, but our proof is new and
simpler. 


\begin{theo}\label{iff} The measure $\mu_p$, $p\in (0,1)$, is an invariant measure of
  the PCA $F$ of parameters $\theta_{00}, \theta_{01}, \theta_{10}, \theta_{11}$ if and only if one of
  the two following conditions is satisfied: 
\begin{eqnarray*}
(i) & (1-p)\;\theta_{00}+p\;
\theta_{01}=(1-p)\;\theta_{10}+p\; \theta_{11}=p \\
(ii) & (1-p)\;\theta_{00}+p\;
\theta_{10}=(1-p)\;\theta_{01}+p\; \theta_{11}=p\:.
\end{eqnarray*}
In particular, a PCA has a (non-trivial) Bernoulli product invariant measure if and only if its parameters satisfy:
\begin{equation}\label{eq-bernou}
\theta_{00}\;(1-\theta_{11})=\theta_{10}\;(1-\theta_{01})
\; \quad \mbox{  or  } \quad \;
\theta_{00}\;(1-\theta_{11})=\theta_{01}\;(1-\theta_{10})\:.
\end{equation}
\end{theo}

\begin{proof} Let us assume that $F$ satisfies condition {\it
    (i)} for some $p\in (0,1)$. Then, the function $g_1$ is given by $g_1(q)=(1-q)\; p +
  q\; p= p$, and $g_0(q)=1-g_1(q)=1-p$. By Proposition
  \ref{formule_vois2}, we have,
\[
\forall \alpha=\alpha_0\cdots \alpha_{n-1}\in\A^n, \quad \mu_p F[\alpha ]
=(1-p)^{|\alpha |_0}p^{|\alpha |_1} = \mu_p[\alpha ] \:.
\]
So $\mu_p$ is an invariant measure. 

Now, assume that the PCA $F$ satisfies condition {\it
    (ii)}. Let us reverse the space direction, that is, let us read the
  configurations from right to left. The same dynamic is now described
  by a new PCA $\widetilde{F}$ defined by the parameters
  $\tilde{\theta}_{00}=\theta_{00}, \tilde{\theta}_{01}=\theta_{10},
  \tilde{\theta}_{10}=\theta_{01},
  \tilde{\theta}_{11}=\theta_{11}$. So, the new PCA satisfies condition
  $(i)$. According to the above, we have $\mu_p
  \widetilde{F}=\mu_p$. Let us reverse the space direction, once
  again. Since the Bernoulli product measure is unchanged, we obtain  $\mu_p
  F=\mu_p$.


Conversely, assume that $\mu_pF=\mu_p$.  It follows from Proposition
\ref{formule_vois2} that for any value of the $\alpha_i$, we must have
$g_{1}(h_{\alpha_{n-1}}(h_{\alpha_{n-2}}(\ldots h_{\alpha_0}(p)\ldots
)))=p$. 
Since $g_1$ is an affine function, there are only two
possibilities: either $g_1$ is the constant function equal to $p$; or
$h_{\alpha_{n-1}}(h_{\alpha_{n-2}}(\ldots h_{\alpha_0}(p)\ldots )))=p$
for all values of $\alpha_0,\ldots , \alpha_{n-1}\in\A$. 

In the first case, observe that
$$g_{1}(q)=q\;
[-(1-p)\;\theta_{00}-p\;\theta_{01}+(1-p)\;\theta_{10}+p\;\theta_{11}]+(1-p)\;
\theta_{00}+p\;\theta_{01}\:.$$  
To get: $\forall q\in [0,1], \  g_1(q)=p$, we must have condition {\it (i)}.

In the second case, we must have $h_0(p)=h_1(p)=p$ and
$g_1(p)=p$. Using $g_0(p)=1-p$ and $g_1(p)=p$, we get:
\begin{eqnarray*}
h_{0}(p)& = & \bigl[(1-p)\; p \; (1- \theta_{01}) + p\;
p\; (1- \theta_{11})\bigr] (1-p)^{-1} \\
h_{1}(p)& = & \bigl[(1-p)\; p \; \theta_{01}+ p\;
p\; \theta_{11} \bigr] p^{-1} \ = \ (1-p)\;
\theta_{01} +p\;\theta_{11}\:.  
\end{eqnarray*}
The equality $h_1(p)=p$ provides the condition $(1-p)\;
\theta_{01} +p\;\theta_{11} =p$. Let us switch to the
equality $h_0(p)=p$. We have:
\begin{eqnarray*}
h_0(p)=p & \iff & (1-p) \; (1- \theta_{01}) + 
p\; (1- \theta_{11}) = 1-p \\ 
& \iff & (1-p) \; \theta_{01} + p \; \theta_{11} = p \:.
\end{eqnarray*}
So, we obtain condition {\it (ii)}.
\end{proof}

To complete Theorem \ref{iff}, let us quote a result from
\cite{vas}. We recall that a PCA has positive rates if: $\forall u
\in \Neighb, \forall a \in \A, \ f(u)(a)>0$. 

\begin{prop}\label{pr-ergodic}
Consider a positive-rates PCA $F$ satisfying condition $(i)$ or $(ii)$, for some $p\in
(0,1)$. Then $F$ is ergodic, that is, $\mu_p$ is the unique invariant
measure of $F$ and for all initial measure $\mu$, the sequence
$(\mu F^n)_{n\geq 0}$ converges weakly to $\mu_p$. 
\end{prop}

Assessing the ergodicity of a PCA is a difficult problem, which is
algorithmically undecidable in general, see \cite{toom00,BMM11}. On
the other hand, a long standing conjecture had
been that any PCA with positive rates is ergodic. However, in 2001,
G\'{a}cs disproved the conjecture by exhibiting a very complex
counter-example with several invariant
measures~\cite{gacs}. In this complicated landscape, Proposition \ref{pr-ergodic} gives a restricted
setting in which ergodicity can be proven. 

\medskip

Observe that Proposition \ref{pr-ergodic} is not true without the
positive-rates assumption. Consider for instance the PCA defined by: 
$\theta_{00}=p/(1-p), \theta_{01}=0,\theta_{10}=0, \theta_{11}=1$ for
some $p\in (0,1/2]$. It satisfies $(i)$ and $(ii)$, but it is not
ergodic since $\delta_{1^{\Z}}$ and $\mu_{p}$ are both invariant. 

\subsection{Transversal PCA}\label{sse-conjug}

We assume that $\mu_p$ is invariant under the action of the PCA, and
we focus on the correlation structure of the space-time diagram obtained
when the initial measure is $\mu_p$. Observe that this space-time
diagram is both space-stationary and time-stationary. By time-stationarity, the
space-time diagram can be extended from $\Z\times \N$ to
$\Z^2$. From now on, we work with this extension. 

Let $(X_{k,n})_{k,n\in \Z\times \Z}$ be a realization of the
stationary space-time diagram.

\begin{center}
\includegraphics[scale=0.75]{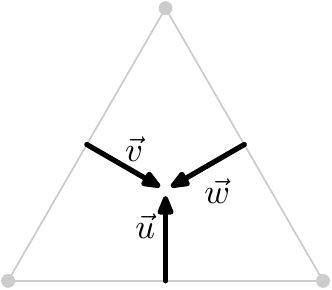}
\end{center}

It is convenient to define the three vectors $\vec{u}, \vec{v},$ and
$\vec{w}$ as in the figure above. 
The PCA generating the space-time diagram is the PCA of direction
$\vec{u}$. In some cases, the space-time diagram when rotated by an angle
of $2\pi/3$ (resp. $-2\pi/3$) still has the correlation structure of a space-time
diagram generated by a PCA of neighborhood $\{0,1\}$. In this case, we say that, in the
original space-time diagram, there is a {\em transversal PCA} of direction
  $\vec{v}$ (resp. $\vec{w}$). 

\begin{prop}\label{ACPdir1}
Under condition {\it (i)}, each line of angle $\pi/3$ of the
space-time diagram is distributed according to $\mu_p$. Moreover,
their correlations are the ones of a \emph{transversal} PCA of
direction $\vec{v}$ and rates given by:
$\vartheta_{00}=\theta_{00}$, $\vartheta_{01}=\theta_{10}$,
$\vartheta_{10}=\theta_{01}$, $\vartheta_{11}=\theta_{11}$.
\end{prop}

To prove Prop. \ref{ACPdir1}, we need two preliminary lemmas. Set $X=X^0$ and $Y=X^1$, so that we have in particular
$(X,Y)\sim (\mu_p, \mu_p F)$. 

\begin{center}
\includegraphics[scale=0.75]{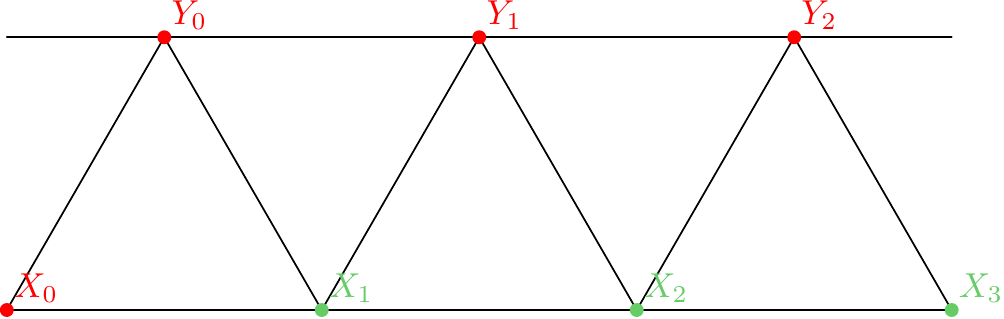}
\end{center}

\begin{lemm}\label{lem1}
Under condition {\it (i)}, the variables $(Y_k)_{k\geq 0}$ are independent of $X_0$, that is, for any $n\geq 0$,
$$\P\bigl(X_0=x_0,(Y_i)_{0\leq i\leq n}=(y_i)_{0\leq i\leq n}\bigr)=\mu_p[x_0]\prod_{i=0}^n\mu_p[y_i]\:.$$
\end{lemm}

\begin{proof} 
The left-hand side can be decomposed into:
$$\sum_{x_1\cdots x_{n+1}\in\{0,1\}^{n+1}}\P\bigl((X_i)_{0\leq i\leq n+1}=(x_i)_{0\leq i\leq n+1},(Y_i)_{0\leq i\leq n}=(y_i)_{0\leq i\leq n}\bigr),$$
which can be expressed with the transition rates of the PCA as follows:
$$\sum_{x_1 \cdots x_{n+1}\in\{0,1\}^{n+1}}\mu_p[x_0]\prod_{i=0}^{n} \mu_p[x_{i+1}] \theta_{x_ix_{i+1}}^{y_i}$$
$$=\mu_p[x_0]\sum_{x_1\in\{0,1\}}\mu_p[x_1]\theta_{x_0x_1}^{y_0}\sum_{x_2\in\{0,1\}}\mu_p[x_2]\theta_{x_1x_2}^{y_1}\
\cdots \ \sum_{x_{n+1}\in\{0,1\}}\mu_p[x_{n+1}]\theta_{x_n
  x_{n+1}}^{y_n}.$$ 
Condition ${\it (i)}$ can be rewritten as: 
\[
\forall a,b,c \in \{0,1\}, \qquad \sum_{b\in\{0,1\}}\mu_p[b]\theta_{ab}^c=\mu_p[c]\: .
\]
Using this, 
and simplifying from the right to the left, we obtain: $\mu_p[x_0]\prod_{i=0}^n\mu_p[y_i]$.\end{proof}

\begin{lemm}\label{lem2}
Under condition {\it (i)}, for any $n\geq 0$,
$$\P(X_0=x_0,X_1=x_1,(Y_i)_{0\leq i\leq n}=(y_i)_{0\leq i\leq n})=\mu_p[x_0]\mu_p[x_1]\theta_{x_0x_1}^{y_0}\prod_{i=1}^n\mu_p[y_i]\:.$$
\end{lemm}


\begin{proof} The proof is analogous. We decompose the left-hand side into:
$$\sum_{x_2\cdots x_{n+1}\in\{0,1\}^{n}}\P((X_i)_{0\leq i\leq n+1}=(x_i)_{0\leq i\leq n+1},(Y_i)_{0\leq i\leq n}=(y_i)_{0\leq i\leq n}),$$
which can be expressed with the transition rates of the PCA as follows:
$$\sum_{x_2\cdots x_{n+1}\in\{0,1\}^{n}}\mu_p[x_0]\prod_{i=0}^{n} \mu_p[x_{i+1}] \theta_{x_ix_{i+1}}^{y_i}$$
$$=\mu_p[x_0]\mu_p[x_1]\theta_{x_0x_1}^{y_0}\sum_{x_2\in\{0,1\}}\mu_p[x_2]\theta_{x_1x_2}^{y_1}\ldots\sum_{x_{n+1}\in\{0,1\}}\mu_p[x_{n+1}]\theta_{x_n
  x_{n+1}}^{y_n}.$$
Using ${\it (i)}$ and simplifying from the right to the left, we get the result. \end{proof}




\begin{proof}[Proof of Proposition \ref{ACPdir1}.] To prove the first
  part of the proposition, it is sufficient to prove that the sequence
  $(X_0^k)_{k\in\Z}$ is i.i.d. For a given $n\in \N$ and a sequence
  $(\alpha_k)_{0\leq k\leq n}$, let us prove recursively that
  $\P((X_0^n)_{0\leq k\leq n}=(\a_k)_{0\leq k\leq
    n})=\mu_p[\alpha_0\cdots \alpha_n].$ For $n=0$, the result is
  straightforward; and for $n=1$, it is a direct consequence of Lemma
  \ref{lem1}. For larger values of $n$, set $A=\P((X_0^k)_{0\leq
    k\leq n}=(\a_k)_{0\leq k\leq n})$, we have:
$$A \ = \sum_{y_1\cdots y_{n-1}\in\{0,1\}^{n-1}}\P\bigl((X_0^k)_{0\leq k\leq
  n}=(\a_k)_{0\leq k\leq n}, (Y_i)_{1\leq i\leq n-1}=(y_i)_{1\leq
  i\leq n-1}\bigr) \:.$$
Since $X^0_0=X_0, X_0^1=Y_0$, it can be rewritten as:
\begin{eqnarray*}
A&=& \sum_{y_1\cdots y_{n-1}\in\{0,1\}^{n-1}}\P\bigl((X_0^k)_{2\leq k\leq n}=(\a_k)_{2\leq k\leq n}\mid X_0=\alpha_0, Y_0=\alpha_1, (Y_i)_{1\leq i\leq n-1}=(y_i)_{1\leq i\leq n-1}\bigl)\\
&& \qquad \qquad \qquad \qquad \times \; \P\bigl(X_0=\alpha_0, Y_0=\alpha_1, (Y_i)_{1\leq i\leq  n-1}=(y_i)_{1\leq i\leq n-1}\bigl)\:.
   \end{eqnarray*}
The law of $(X_0^k)_{2\leq k\leq n}$ conditionally to
$(X_0,(Y_i)_{0\leq i\leq n-1})$ is equal to the law of $(X_0^k)_{2\leq
  k\leq n}$ conditionally to $(Y_i)_{0\leq i\leq n-1}$. 
 Also, 
using Lemma \ref{lem1}, we have: $\P\bigl(X_0=\alpha_0, Y_0=\alpha_1, (Y_i)_{1\leq i\leq
  n-1}=(y_i)_{1\leq i\leq n-1}\bigl)= \mu_p[\alpha_0] \;
\P(Y_0=\alpha_1, (Y_i)_{1\leq i\leq n-1}=(y_i)_{1\leq i\leq
  n-1})$. Coupling these two points, we get:  
\begin{eqnarray*}
A & = & \sum_{y_1\cdots y_{n-1}\in\{0,1\}^{n-1}}\P\bigl((X_0^k)_{2\leq
  k\leq n}=(\a_k)_{2\leq k\leq n}\mid Y_0=\alpha_1, (Y_i)_{1\leq i\leq
  n-1}=(y_i)_{1\leq i\leq n-1}\bigl) \\
&& \qquad \qquad \qquad \qquad \times \; \mu_p[\alpha_0] \; \P(Y_0=\alpha_1, (Y_i)_{1\leq
  i\leq n-1}=(y_i)_{1\leq i\leq n-1}) \\
&&\\
& = & \mu_p[\alpha_0]\; \P\bigl((X_0^k)_{1\leq k\leq n}=(\a_k)_{1\leq
  k\leq n}\bigr) \:.
\end{eqnarray*}
By induction, we obtain the result. 

\begin{center}
\includegraphics[scale=0.6]{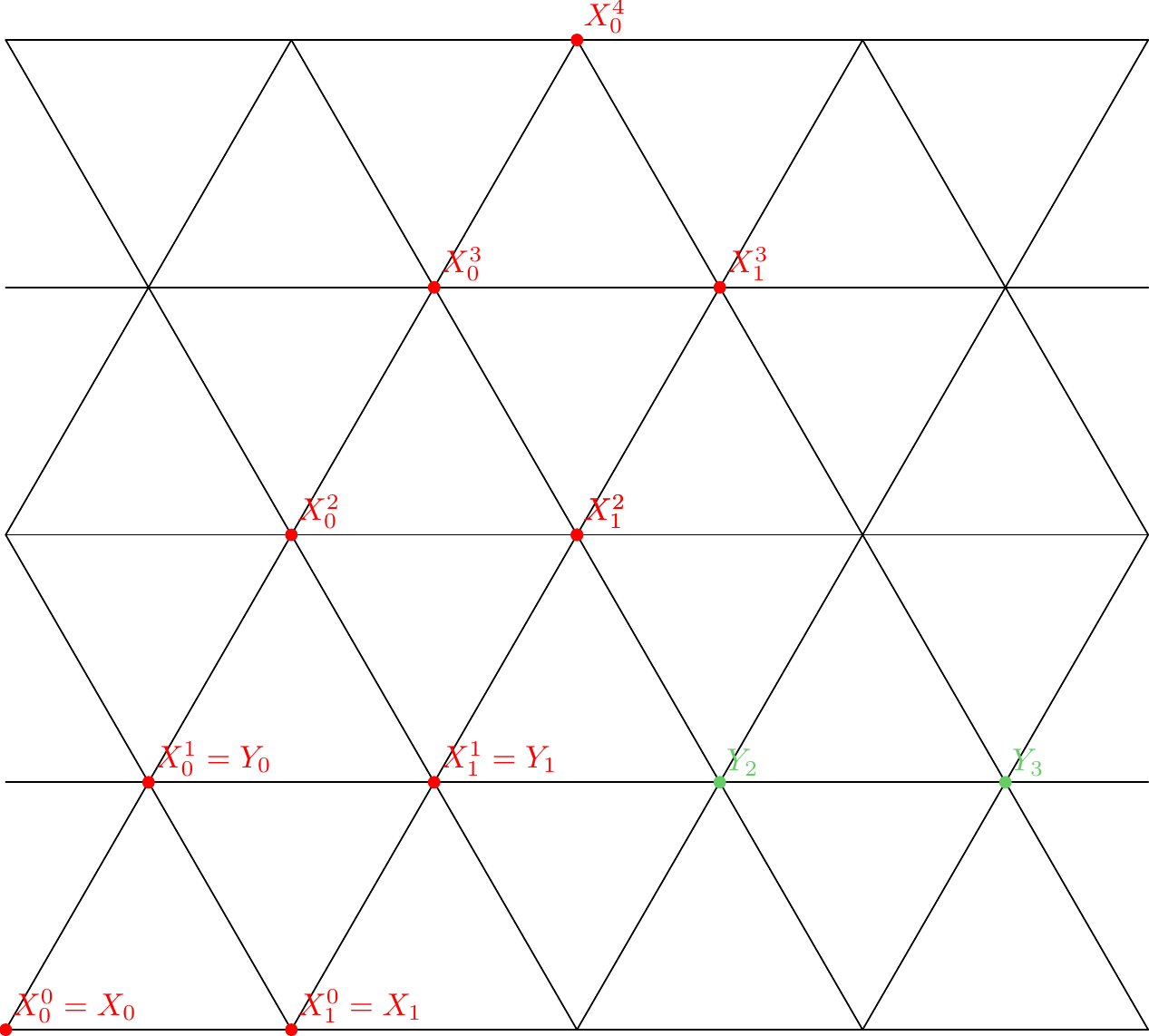}
\end{center}

The second part of the proposition consists in proving that 
\begin{equation}\label{eq-equiv}
\P\bigl( (X^k_1)_{0\leq k\leq n}=(\beta_k)_{0\leq k\leq n}\mid
(X^k_0)_{0\leq k\leq n+1}=(\alpha_k)_{0\leq k\leq n+1} \bigr)=\prod_{k=0}^n
\vartheta_{\alpha_{k+1}\alpha_{k}}^{\beta_k}\:.
\end{equation}

We prove the result recursively. For $n=0$, set $A= \P(X_1=\beta_0
\mid Y_0=\a_1, X_0=\a_0 )$. We want to prove that $A=
\vartheta^{\beta_0}_{\a_1\a_0}$. 
Using the first part of the proposition, we have: 
\begin{eqnarray*}
A & = & \P(Y_0=\a_1 \mid
X_0=\a_0,X_1=\beta_0 ) \; \P(X_0=\a_0,X_1=\beta_0 ) \;
\P(X_0=\a_0,Y_0=\a_1)^{-1} \\
& = & \theta^{\a_1}_{\a_0\beta_0}\; 
\mu_p[\a_0]\;\mu_p[\beta_0]\; \mu_p[\a_0]^{-1}\;\mu_p[\a_1]^{-1} \quad = \quad
\theta^{\a_1}_{\a_0\beta_0} \;\mu_p[\beta_0]\;\mu_p[\a_1]^{-1}  \:.
\end{eqnarray*}
If $\a_1=\beta_0=u$, we get $A= \theta^{u}_{\a_0 u} =
\vartheta^{u}_{u\a_0}$. Assume that $\a_1\not=\beta_0$. Condition
$(i)$ can be rewritten as: 
\begin{equation}\label{eq-rewrite}
\mu_p[\beta_0]\; \theta^{\a_1}_{\a_0\beta_0} +
\mu_p[\a_1]\; \theta^{\a_1}_{\a_0\a_1} = \mu_p[\a_1]\:.
\end{equation}
Dividing by $\mu_p[\a_1]$, we get: 
\[
A= \theta^{\a_1}_{\a_0\beta_0} \; \mu_p[\beta_0]\; \mu_p[\a_1]^{-1} = 1 -
\theta^{\a_1}_{\a_0\a_1} =\theta^{\beta_0}_{\a_0\a_1} =\vartheta^{\beta_0}_{\a_1\a_0}\:.
\]

For larger $n$, it is convenient to prove next equality, which is
equivalent to \eref{eq-equiv}:
$$\P\bigl((X^k_0)_{0\leq k\leq n+1}=(\a_k)_{0\leq k\leq n+1},(X^k_1)_{0\leq
  k\leq n}=(\beta_k)_{0\leq k\leq n}\bigr)\ =\ \mu_p[\a_{n+1}]\prod_{k=0}^{n}
\mu_p[\a_k]\vartheta_{\alpha_{k+1}\alpha_{k}}^{\beta_k}.$$ 
The left-hand side can be decomposed into:
$$\sum_{y_2\cdots y_{n}\in\{0,1\}^{n-1}}\P\bigl((X^k_0)_{0\leq k\leq
  n+1}=(\a_k)_{0\leq k\leq n+1},(X^k_1)_{0\leq k\leq
  n}=(\beta_k)_{0\leq k\leq n},(Y_i)_{2\leq i\leq n}=(y_i)_{2\leq
  i\leq n}\bigr)\:.$$ 
Let us decompose each term of the sum, conditioning by the values of $X_0,X_1,Y_0,$ and $Y_1$.
We have: 
{\footnotesize $$\P\bigl((X^k_0)_{2\leq k\leq n+1}=(\a_k)_{2\leq k\leq n+1},(X^k_1)_{2\leq k\leq n}=(\beta_k)_{2\leq k\leq n}\mid (X_0,X_1,Y_0,Y_1)=(\a_0,\beta_0,\a_1,\beta_1),(Y_i)_{2\leq i\leq n}=(y_i)_{2\leq i\leq n}\bigr)$$
$$=\P\bigl((X^k_0)_{2\leq k\leq n+1}=(\a_k)_{2\leq k\leq n+1},(X^k_1)_{2\leq
  k\leq n}=(\beta_k)_{2\leq k\leq n}\mid
(Y_0,Y_1)=(\a_1,\beta_1),(Y_i)_{2\leq i\leq n}=(y_i)_{2\leq i\leq
  n}\bigr)\:.$$}

and using Lemma \ref{lem2}, and the equality
$\mu_p[\beta_0]\theta_{\a_0\beta_0}^{\a_1}=\mu_p[\a_1]\vartheta_{\a_1\a_0}^{\beta_0}$  (see \eref{eq-rewrite}):
$$\P\bigl( (X_0,X_1,Y_0,Y_1)=(\a_0,\beta_0,\a_1,\beta_1),(Y_i)_{2\leq
  i\leq n}=(y_i)_{2\leq i\leq n} \bigr)\qquad \qquad \qquad $$
\begin{eqnarray*}
& = & \mu_p[\a_0]\mu_p[\beta_0]\theta_{\a_0\beta_0}^{\a_1}\;
\P\bigl(Y_1=\beta_1,(Y_i)_{2\leq i\leq n}=(y_i)_{2\leq i\leq n}\bigr) \\
&  = & \mu_p[\a_0]\mu_p[\a_1]\vartheta_{\a_1\a_0}^{\beta_0}\;
\P\bigl(Y_1=\beta_1,(Y_i)_{2\leq i\leq n}=(y_i)_{2\leq i\leq n}\bigr) \\
& = & \mu_p[\a_0]\vartheta_{\a_1\a_0}^{\beta_0}\;
\P\bigl((Y_0,Y_1)=(\a_1,\beta_1),(Y_i)_{2\leq i\leq n}=(y_i)_{2\leq
  i\leq n}\bigr)\:. 
\end{eqnarray*}
Assembling the pieces together, we obtain:
$$\P\bigl((X^k_0)_{0\leq k\leq n+1}=(\a_k)_{0\leq k\leq n+1},(X^k_1)_{0\leq k\leq n}=(\beta_k)_{0\leq k\leq n}\bigr)$$
$$=\mu_p[\a_0]\vartheta_{\a_1\a_0}^{\beta_0}\; \P\bigl((X^k_0)_{1\leq k\leq n+1}=(\a_k)_{1\leq k\leq n+1},(X^k_1)_{1\leq k\leq n}=(\beta_k)_{1\leq k\leq n}\bigr)\:.$$
We conclude the proof by induction.
\end{proof}

\begin{cor}\label{cor1} Under condition {\it (i)}, all the lines of the space-time diagram except possibly those of angle $2\pi/3$ consist of i.i.d. random variables.
\end{cor}

\begin{proof} The previous proposition claims that the lines of angle
  $\pi/3$ are i.i.d. Lemma \ref{outside} provides the result for
  the lines of angle in $[0,\pi/3)\cup (2\pi/3,\pi]$. The
  angles in $(\pi/3,2\pi/3)$ correspond to lines that are outside the
  dependence cones of the transversal PCA, so we obtain the result by
  applying again Lemma \ref{outside} for the transversal
  PCA. \end{proof} 

In the same way, one can prove the following.

\begin{prop}\label{ACPdir2}
Under condition {\it (ii)}, the lines of angle $2\pi/3$ of the
space-time diagram are distributed according to $\mu_p$ and their
correlations are those of a \rm{transversal} PCA of direction
$\vec{w}$ and rates given by $\vartheta_{00}=\theta_{00}$,
$\vartheta_{11}=\theta_{11}$ and $\vartheta_{01}=\theta_{10}$,
$\vartheta_{10}=\theta_{01}$. 
\end{prop}

\begin{cor}\label{cor2} Under condition {\it (ii)}, all the lines of
  the space-time diagram except possibly the ones of angle $\pi/3$
  consist of i.i.d. random variables. 
\end{cor}

For a PCA satisfying {\it
  (i)} (resp.  {\it
  (ii)}), the lines of angle $2\pi/3$ (resp. $\pi/3$) are not i.i.d., except if the PCA
also satisfies condition {\it (ii)} (resp. {\it
  (i)}). The distribution of the lines of
angle $2\pi/3$ (resp. $\pi/3$) does not 
necessary have a Markovian form either. For example, if
$\theta_{00}=\theta_{01}=1/2$ and $\theta_{10}=0,\theta_{11}=1$
(condition {\it (i)} is satisfied with $p=1/2$), one can check that
$\P(X_0^0=0,X_{-1}^1=0,X_{-2}^2=0)=19/64$ which is different $\P(X_0^0=0) \P(X_{-1}^1=0\mid
X_0^0=0)\P(X_{-2}^2=0\mid X_{-1}^1=0)= (1/2)(3/4)^2$. 

\medskip

It is an open problem to know if under condition {\it (i)} (resp. {\it
  (ii)}), it is possible to give an explicit description of the
distribution of the lines of angle $2\pi/3$ (resp. $\pi/3$). 

\section{Non-i.i.d. random field with every line i.i.d.}\label{spatial}

We now concentrate on PCA
satisfying {\bf both} conditions $(i)$ and $(ii)$ for some $p\in (0,1)$. We
consider the stationary space-time
diagram associated with $\mu_p$, and we
still denote it by $(X^n_k)_{k,n\in \Z}$.


\subsection{All the lines are i.i.d.}

For a given $p\in (0,1)$, conditions  {\it (i)} and {\it (ii)} are
both satisfied if and only if:
\begin{equation}\label{eq-i+ii}
\exists s \in \Bigl[\frac{2p-1}{p},\frac{p}{1-p}\Bigr], \qquad
\theta_{00}=\frac{p(1-s)}{1-p}, \ \theta_{01}=\theta_{10}=s, \ 
\theta_{11}=1-\frac{(1-p)s}{p} \:.
\end{equation}


\begin{exam}\rm
For any value of $p\in (0,1)$, the choice $s=p$ is allowed. In that
case, the transition rates $\theta_{ij}$ are all equal to $p$ and the stationary
random field is i.i.d., there is no dependence in the
space-time diagram. \end{exam}



\begin{exam}\rm\label{ex:sum}
If $p=1/2$, every choice of $s\in[0,1]$ is valid and the corresponding
PCA has the transition function $f(x,y)=s\; \delta_{x+y \mod 2}+(1-s)\; \delta_{x+y+1 \mod 2}$.
\end{exam}

\begin{center}
\includegraphics[scale=0.25]{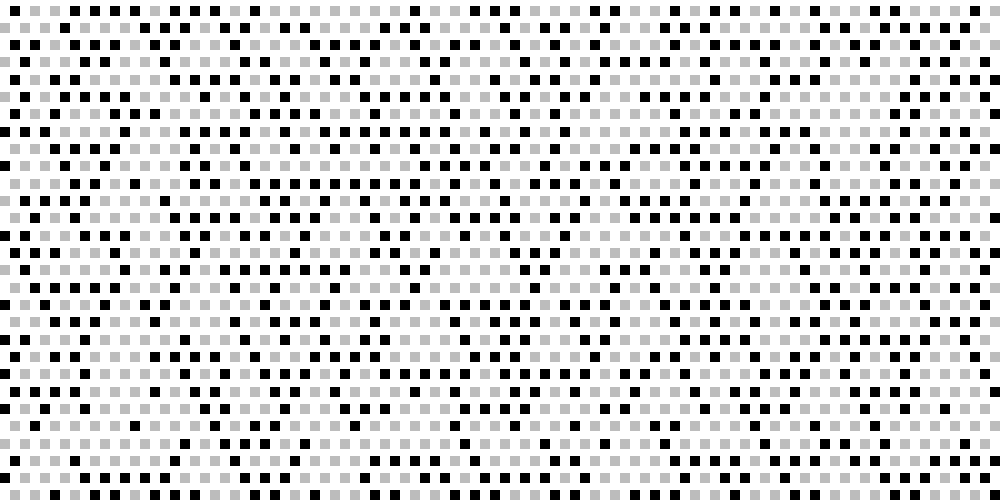}

An example of space-time diagram for $p=1/2$ and $s=3/4$
\end{center}

\begin{exam}\rm\label{ex:triang}
For any value of $p\in (0,1/2]$, it is possible to set $s=0$ and then,
$\theta_{01}=\theta_{10}=0, \theta_{11}=1$, and
$\theta_{00}=p/(1-p)$. This PCA forbids the elementary triangles
pointing up that have exactly one vertex labeled by a $0$.  
\end{exam}

\begin{center}
\includegraphics[scale=0.25]{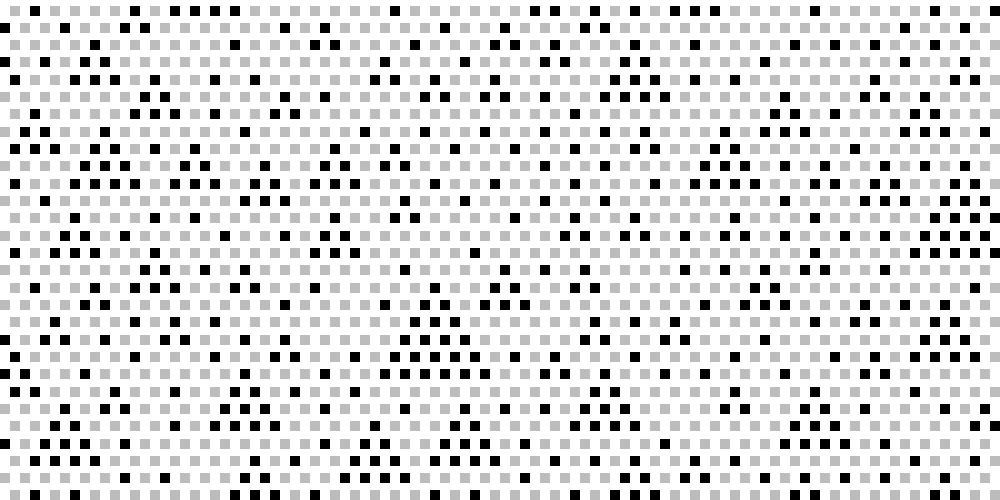}

An example of space-time diagram for $p=1/3$ and $s=0$
\end{center}

Next proposition is a direct consequence of Corollaries \ref{cor1} and
\ref{cor2}.

\begin{prop}\label{pr-all} Consider a PCA satisfying \eref{eq-i+ii}. 
Every line of the stationary space-time diagram
  consists of i.i.d. random variables. In particular, any two
  different variables are independent.  
\end{prop}

\subsection{Equilateral triangles pointing up are correlated}
\label{sse-equilateral}

We have seen that all the lines of the space-time diagram are i.i.d. But the whole space-time diagram is
i.i.d. if and only if $s=p$. Indeed, if $s\neq p$, the random variable $X^{n+1}_k$ is not independent of
$(X^{n}_k,X^n_{k+1})$; in words, the
three variables of an elementary
triangle pointing up are correlated.  Precisely, the triple 
$(X^{n}_k,X^n_{k+1}, X^{n+1}_k)$ consists of random variables
which are: (1) identically distributed; (2)
pairwise independent; (3) globally dependent if $s\neq p$. 
The ``converse'' holds. 

\begin{prop}
Let $\nu$ be a law on $\{0,1\}^3$ such that the three marginals on
$\{0,1\}^2$ are i.i.d. Assume that $\nu$ is non degenerated ($\nu \neq
\delta_{000}, \nu \neq \delta_{111}$). Then $\nu$ can be realized as the law of an
``elementary triangle pointing up'' in the stationary space-time
diagram of exactly one PCA satisfying \eref{eq-i+ii}. 
\end{prop}



\begin{proof} 
Consider $(X_0,X_1,Y_0)\sim \nu$. Assume that the common law of
$X_0,X_1,$ and $Y_0$ is $\ber_p$. 
By the pairwise independence, we have:
\begin{eqnarray*}
\P(X_0=1,X_1=0,Y_0=0)& = & \P(X_1=0,Y_0=0) - \P(X_0=0,X_1=0,Y_0=0)\\
& = & (1-p)^2 -  \P(X_0=0,X_1=0,Y_0=0) \:.
\end{eqnarray*}
We obtain:
\begin{eqnarray*}
&&\P(X_0=1,X_1=0,Y_0=0)=\P(X_0=0,X_1=1,Y_0=0)=\P(X_0=0,X_1=0,Y_0=1) \\
&&\P(X_0=0,X_1=1,Y_0=1)=\P(X_0=1,X_1=0,Y_0=1)=\P(X_0=1,X_1=1,Y_0=0)\:.
\end{eqnarray*}
Set $q_0 = \P(X_0=1,X_1=0,Y_0=0)$ and $q_1=
\P(X_0=0,X_1=1,Y_0=1)$. We have:
\begin{equation*}
\P(X_0=0,X_1=0,Y_0=0) =   (1-p)^2- q_0, \qquad 
\P(X_0=1,X_1=1,Y_0=1)  =    p^2 - q_1\:.
\end{equation*}
Furthermore:
\begin{equation*}
q_0+q_1= \P(X_0=0,X_1=0,Y_0=1) + \P(X_0=1,X_1=0,Y_0=1) =
\P(X_1=0,Y_0=1) = p(1-p) \:.
\end{equation*}
Using the above, and expressing everything as a function of $p$ and $q_1$, we get:
\begin{eqnarray*}
\P(Y_0=1\mid X_0=0,X_1=0) & = & (p(1-p)-q_1)/(1-p)^2 \\
\P(Y_0=1\mid X_0=0,X_1=1) & = & q_1/(p(1-p)) \\
\P(Y_0=1\mid X_0=1,X_1=0) & = & q_1/(p(1-p)) \\ 
\P(Y_0=1\mid X_0=1,X_1=1) & = & 1-q_1/p^2 \:.
\end{eqnarray*}
By setting $\theta_{ij} = \P(Y_0=1\mid X_0=i,X_1=j)$ and $s=q_1/(p(1-p))$, we recover
exactly \eref{eq-i+ii}. 
\end{proof}

\begin{prop}\label{pr-triangle} Consider a PCA satisfying \eref{eq-i+ii} with $s\neq p$. 
The correlations between  three random variables that form an equilateral
  triangle pointing up decrease exponentially in
  function of the size of the triangle. 
\end{prop}

\begin{proof} Let us consider the random field $(X_{2k}^{2n})_{k,n\in
    \Z}$. Observe that all its random variables are distributed according to $\ber_p$, and that each
  line consists of i.i.d. random variables. 
  Moreover, for any $a<b$, the variables $(X^{2n+2}_{2k})_{a\leq k\leq b}$ are independent conditionally to the variables $(X^{2n}_{2k})_{a\leq k\leq b+1}$.
 Thus, this ``extracted'' random field corresponds to the space-time diagram of a new PCA, having a
  neigborhood of size $2$ and satisfying \eref{eq-i+ii} for the
  same value of $p$. To know its transition rates
  $\theta^{(2)}_{ij}=\P(X_0^2=1 \mid X_0^0=i, X_2^0=j)$, it is 
  enough to compute $\theta^{(2)}_{10}=\theta^{(2)}_{01}$. We 
 denote this value by $\phi(s)$, since it is a function of
  $s=\theta_{01}=\theta_{10}$. 

Summing on all possible values of
  $X_1^0, X^1_0, X^1_1$ (we first consider the case $X_1^0=1$ and then
  the one $X_1^0=0$), we get: 
$$\phi(s)
=p\;[\theta_{01}\theta_{11}\;
\theta_{11}+(1-\theta_{01})\theta_{11}\;
\theta_{01}+\theta_{01}(1-\theta_{11})\;
\theta_{10}+(1-\theta_{01})(1-\theta_{11})\; \theta_{00}]$$ 
$$+(1-p)\;[\theta_{00}\theta_{01}\;
\theta_{11}+(1-\theta_{00})\theta_{01}\;
\theta_{01}+\theta_{00}(1-\theta_{01})\;
\theta_{10}+(1-\theta_{00})(1-\theta_{01})\; \theta_{00}].$$ 
Replacing the coefficients $\theta_{ij}$ by their expression in function of $p$ and $s$ and simplifying the result, we obtain:
$$\phi(s)=p+{(s-p)^3\over p(1-p)}.$$

We proceed similarly for the random field $(X_{2^i k}^{2^i n})_{k,n\in
    \Z}$.
The coefficient $\theta^{(2^i)}_{01}=\P(X_0^{2^i}=1\mid X_0^0=0, X_{2^i}^0=1)$ is equal to $\phi^i(s)$, which satisfies:
$$\phi^i(s)-p={(s-p)^{3^i}\over (p(1-p))^{{3^{i}-1\over
      2}}}=\sqrt{p(1-p)}\Big({s-p\over \sqrt{p(1-p)}}\Big)^{3^i}\:.$$

Similar computations can be performed for equilateral triangles
pointing up of other sizes. 
The decay of correlation for equilateral triangles pointing up is exponential in function of their size. \end{proof}


Next lemma will allow us to characterize completely the triples
of random variables that are not independent. 

\begin{lemm}\label{pt_indep}
Consider a PCA satisfying \eref{eq-i+ii}. The variable $X_0^0$ is independent of $(X_k^n)_{k\in\Z,n\in\N\setminus\{0\}}$.
\end{lemm}

\begin{proof} Set $X=X^0$ and $Y=X^1$. 
It is sufficient to prove that  $X_0$ is independent of
  $(Y_k)_{k\in\Z}$. But
  $(Y_k)_{k\geq 0}$ and $(Y_k)_{k<0}$ are independent conditionally to
  $X_0$, so that we can conclude with Lemma \ref{lem1} and its
  analog for condition {\it (ii)}.\end{proof} 

\begin{prop} Consider a PCA satisfying \eref{eq-i+ii} with $s\not=
  p$. Three random variables of the stationary space-time diagram are correlated
  if and only if they form an equilateral triangle pointing up.  
\end{prop}

\begin{proof} Three
  variables that form an equilateral triangle pointing up are
  correlated, see the proof of Proposition \ref{pr-triangle}. Let us now consider three variables $(Z_1,Z_2,Z_3)$ that
  do not constitute such a triangle. Then, if we consider the smallest
  equilateral triangle pointing up that contains them, there is an edge of that
  triangle that contains exactly one of these variables. By rotation
  of angle $2\pi/3$ or translation of the diagram, one can assume that this
  edge is the horizontal one and that it contains the variable
  $Z_1$, and not the variables $Z_2, Z_3$. 
  Now, using Lemma \ref{pt_indep}, we obtain that $Z_1$
  is independent of $(Z_2,Z_3)$. But since $Z_2$ and $Z_3$ are
  independent, the three variables $(Z_1,Z_2,Z_3)$ are independent. 
\end{proof}

There are subsets of four variables that do not contain equilateral
triangles pointing up and that are correlated. It is the case in
general of $(X_0,X_2,Y_0,Y_1)$. Let us consider for instance the PCA
of Example \ref{ex:triang}. The event $(X_0,X_2,Y_0,Y_1)=(0,1,1,1)$
has probability zero, since whatever the value of $X_1$, the
space-time diagram would have an elementary triangle pointing up with
exactly one zero. 

\begin{center}
\includegraphics[scale=0.75]{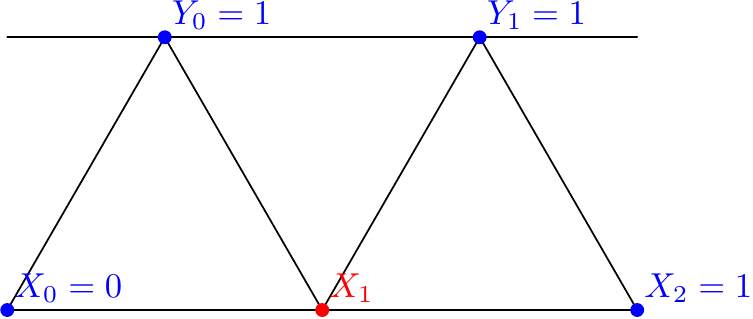}
\end{center}

\subsection{Incremental construction of the random field}

Let us show how to construct
incrementally the stationary space-time diagram of a PCA satisfying
conditions {\it (i)} and {\it (ii)}, using two elementary
operations. 

\medskip

Consider a PCA satisfying {\it (i)} and {\it (ii)} for some $p\in
(0,1)$. 
Let $S\subset\Z^2$ be the finite set of points of the space-time
diagram that has been constructed at some step. Initially
$S=\{(0,0)\}$ and $X_0^0\sim \ber_p$.

\begin{itemize}

\item If $(i,n),(i+1,n)\in S, (i,n+1)\not\in S$, and $\D(i,n+1)\cap S=\emptyset$. Choose
  $X_i^{n+1}$ knowing $(X_i^n , X_{i+1}^n)$  according to
  the law of the PCA.  

If $(i,n),(i,n+1)\in S, (i+1,n)\not\in S$, and if no point of the
dependence cone of $(i+1,n)$ with respect to the transversal
PCA of direction $\vec{v}$ belongs to $S$: choose $X_{i+1}^{n}$
knowing $(X_{i}^{n+1}, X_i^n) $  according to the law of the
transversal PCA  of direction $\vec{v}$. 

If $(i,n+1),(i+1,n)\in S, (i,n)\not\in S$, and if no point of the
dependence cone of $(i,n)$ with respect to the transversal
PCA of direction $\vec{w}$ belongs to $S$: choose $X_{i}^{n}$
knowing $(X_{i+1}^n , X_{i}^{n+1}) $  according to the law of the
transversal PCA  of direction $\vec{w}$. 

\item If $(i,n)\not\in S$, and if $(j,m)\in S$ implies $m<n$: choose
  $X_i^n$ according to $\ber_{p}$ and  independently of the variables
  $X_j^m, (j,m)\in S$. 

If $(i,n)\not\in S$, and if $(j,m)\in S$ implies $j<i$: choose
  $X_i^n$ according to $\ber_{p}$ and  independently of the variables
  $X_j^m, (j,m)\in S$. 

If $(i,n)\not\in S$, and if $(j,m)\in S$ implies $j+m>i+n$: choose
  $X_i^n$ according to $\ber_{p}$ and  independently of the variables
  $X_j^m, (j,m)\in S$. 
\end{itemize}

By applying the above rules in the order illustrated by the figure
below, one  can progressively build the
stationary space-time diagram of the PCA.  
Indeed the rules enlarge $S$ in such a way that, at each step, the
variables of $S$ have the same distribution as the corresponding
finite-dimensional marginal of the stationary space-time diagram. 
This is proved by Lemmas \ref{lem:cone} and \ref{pt_indep}. 

\medskip

On the figure, the labelling of the nodes corresponds to the step at which the
corresponding variable is computed (after the three variables of the grey triangle). An arrow pointing to a variable means that it has been
constructed according to the PCA of the direction of the arrow (first rule). The 
nodes labelled by $\amalg$ are the ones which have been constructed
by independence (second rule).

\begin{center}
\includegraphics[scale=0.80]{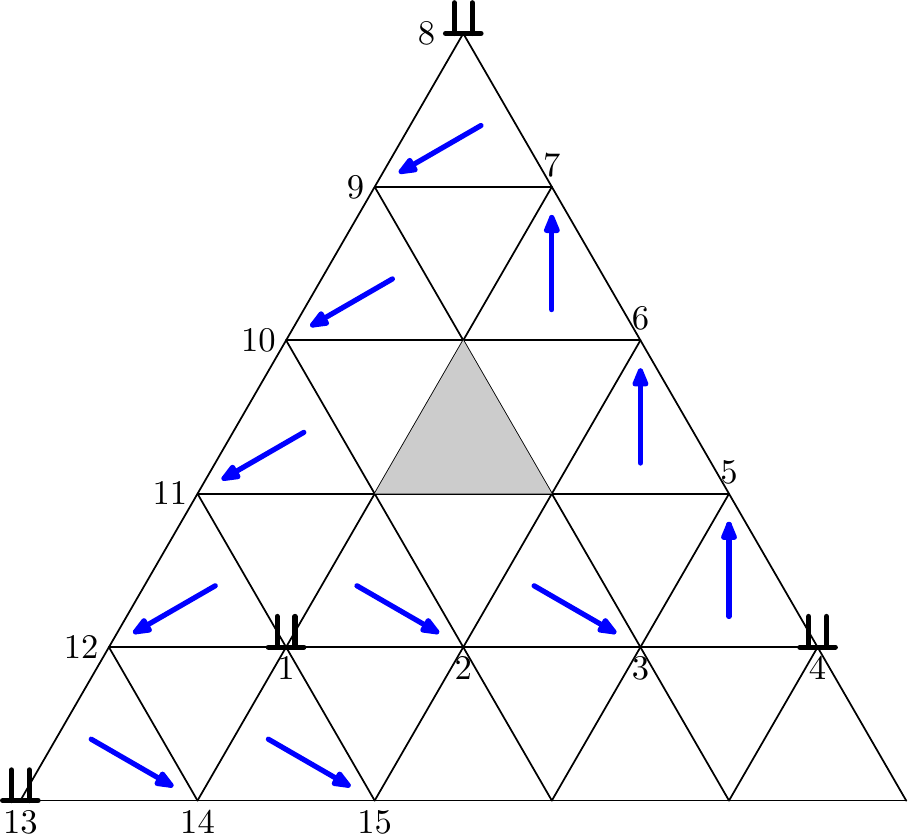}
\end{center}

\section{Extensions}\label{se:extensions}

We consider two types of extensions. First,  PCA with an
alphabet and neighborhood of size 2 but having a Markovian invariant
measure. Second, PCA having a Bernoulli product invariant measure but
with a general alphabet and neighborhood. 

\subsection{Markovian invariant measures}

Markovian measures are a natural extension of Benoulli product
measures. 
In a nutshell, the tools of Section
\ref{cond} can be extended to find conditions for having a Markovian
invariant measure, but the spatial properties presented in Section
\ref{spatial} do not remain.  


\begin{defi}\label{de-mark}
Consider $a,b \in (0,1)$. The {\em Markovian measure} on
$\{0,1\}^{\Z}$ of {\em transition matrix}
$$Q= 
 \left(
  \begin{array}{ c c }
     1-a & a \\
     1-b & b
  \end{array} \right)
$$
is the measure $\nu_Q$ defined on cylinders by:
$$\forall x=x_m\cdots x_n, \qquad  \nu_Q[x]=\pi_{x_m}\;
\prod_{i=m}^{n-1} Q_{x_i,x_{i+1}},$$
where $\pi=(\pi_0, \pi_1)$ is such that $\pi Q=\pi$, $\pi_0+\pi_1=1$, that is,
$\pi_0= (1-b)/(1-b+a)$ and $\pi_1= a/(1-b+a)$. 
\end{defi}

The Markovian measure $\nu_Q$ is space-stationary. 
If $a=b$, then $\nu_Q=\mu_a$, the Bernoulli product measure of parameter
$a$. 

\medskip

Let us fix the PCA, that is, the parameters
$(\theta_{00},\theta_{01},\theta_{10},\theta_{11})$ and assume that
\eref{eq-condition} holds. 
Let us fix the parameters $a$ and $b$ in $(0,1)$ (defining  $Q$ and
$\pi$ as in Definition \ref{de-mark}). We
introduce the 
analogs of the functions defined in \eref{eq-g} and \eref{eq-h}. 

For $\alpha\in\{0,1\}$, define the function:
\begin{equation}\label{eq-g2}
\begin{array}{cccl}
g_{\alpha}: & [0,1] & \longrightarrow &(0,1) \\
&  r & \longmapsto & (1-r)\; (1-a)\;\theta_{00}^{\alpha}+(1-r)\;
a\;\theta_{01}^{\alpha}+r\;
(1-b)\;\theta_{10}^{\alpha}+r\; b\;\theta_{11}^{\alpha} \:.
\end{array}
\end{equation}
In words, $g_{\alpha}(r)$ is the probability that $Y_0=\alpha$ if the
law of $(X_0,X_1)$ is given by $\P(X_0=x_0,X_1=x_1)=r_{x_0}\;
Q_{x_0,x_1}$ with $r_0=1-r$ and $r_1=r$. With condition \eref{eq-condition} on
the parameters, we have $g_{\alpha}(r)\in (0,1)$ for all $r$. Observe
also that: $g_0(r) +
g_1(r) = 1$. 

For 
$\alpha\in\{0,1\}$,  we also define the function:
\begin{equation}\label{eq-h2}
\begin{array}{cccl}
h_{\alpha}: & [0,1] & \longrightarrow & [0,1]\\
 & r & \longmapsto & \bigl[ (1-r)\; a \;
 \theta_{01}^{\alpha}+ r\; b\; \theta_{11}^{\alpha} \bigr] g_{\alpha}(r)^{-1}\:.
\end{array}
\end{equation}
In words, $h_{\alpha}(r)$ is the probability to have $X_1 =1$
conditionally to $Y_0=\alpha$ if
$(X_0,X_1)$ is distributed according to the above law.

\begin{prop}\label{formule_vois2markov} 
Consider the Markovian measure $\nu_Q$ and the PCA $F$ as above. For
$\alpha_0\cdots \alpha_{n-1}\in \A^n$, the probability of the cylinder
$[\alpha_1\cdots \alpha_n]$ under $\nu_Q F$ is given by: 
$$\nu_Q F [\alpha_0 \cdots \alpha_{n-1}] = g_{\alpha_0}(\pi_1) \; \prod_{i=1}^{n-1}
g_{\alpha_i}(h_{\alpha_{i-1}}(h_{\alpha_{i-2}}(\ldots
h_{\alpha_0}(\pi_1)\ldots )))\:.$$ 
\end{prop}

Using Proposition \ref{formule_vois2markov}, we obtain sufficient conditions for having a  Markovian
invariant measure. This provides a new proof of
a result mentioned in \cite{toombook} and first published in
\cite{bel}.

\begin{theo}\label{theo:markov}
A PCA has a Markovian invariant measure if its parameters satisfy:
\begin{equation}\label{eq-mbm}
\theta_{00}\;\theta_{11}\;(1-\theta_{01})\;(1-\theta_{10})=\theta_{01}\;\theta_{10}\;(1-\theta_{00})\;(1-\theta_{11})\:,
\end{equation}
and $(\theta_{00},\theta_{01}), (\theta_{10},\theta_{11}) \not\in
  \{(0,0), (1,1)\}$ or $(\theta_{00},\theta_{10}), (\theta_{01},\theta_{11}) \not\in
  \{(0,0), (1,1)\}$.
\end{theo}

\begin{proof}
We treat the case $(\theta_{00},\theta_{01}), (\theta_{10},\theta_{11}) \not\in
  \{(0,0), (1,1)\}$ (for which Prop. \ref{formule_vois2markov}
  holds). The case $(\theta_{00},\theta_{10}), (\theta_{01},\theta_{11}) \not\in
  \{(0,0), (1,1)\}$ can be treated by reversing the space-direction. 

Let us assume that the following conditions are satisfied:
\begin{enumerate}
\item for $\alpha\in \{0,1\}$, $g_{\alpha}(\pi_{1})=\pi_{\alpha}$;
\item for $\alpha\in \{0,1\}$, there exists $c_{\alpha}\in [0,1]$ such
  that: $\forall r, \ h_{\alpha}(r)= c_{\alpha}$; 
\item for $\alpha, \beta \in \{0,1\}$, $g_{\beta}(c_{\a})=Q_{\a,\beta}$.
\end{enumerate}
Then, by a direct application of Proposition \ref{formule_vois2markov}, the measure $\nu_Q$ is invariant. When are these conditions fulfilled?


For $\a=1$, condition 2 tells us that there exists $c_{1}\in[0,1]$ such that for any $r\in [0,1]$,
$$(1-r)\; a \; \theta_{01}+ r\; b\; \theta_{11} =
c_1 \big((1-r)\; (1-a)\;\theta_{00}+(1-r)\;
a\;\theta_{01}+r\; (1-b)\;\theta_{10}+r\;
b\;\theta_{11}\big).$$ 

This is the case if and only if:
\begin{equation*}
a\;\theta_{01}=c_1((1-a)\;\theta_{00}+ a\;\theta_{01}), \qquad 
b\;\theta_{11}=c_1((1-b)\;\theta_{10}+ b\;\theta_{11})\:.
\end{equation*}
Thus, condition 2 for $\a=1$ is equivalent to: 
\begin{equation}\label{d1}
a \; (1-b) \;\theta_{01}\;\theta_{10}=(1-a)\;b\;\theta_{00}\;\theta_{11}\:.
\end{equation}
In the same way, condition 2 for $\a=0$ is equivalent to: 
\begin{equation}\label{d2}
a \; (1-b)
\;(1-\theta_{01})\;(1-\theta_{10})=(1-a)\;b\;(1-\theta_{00})\;(1-\theta_{11})\:.
\end{equation}
Eliminating $a$ and $b$ in (\ref{d1}) and (\ref{d2}), we obtain the
relation \eref{eq-mbm} for the parameters of the PCA. 


\medskip

Conversely, let us assume that relation \eref{eq-mbm} holds. We will
prove that there exists $a,b \in (0,1)$ such that the three above conditions are satisfied. 

First observe that \eref{d1} holds if and only if \eref{d2} holds. 
So, we have a first relation to be satisfied by the parameters $a,b
\in (0,1)$ which is \eref{d1}. Under this relation, condition 2 is
satisfied with:
\begin{equation}\label{c0}
c_0={a\;(1-\theta_{01})\over (1-a)\;(1-\theta_{00})+
  a\;(1-\theta_{01})}={b\;(1-\theta_{11})\over (1-b)\;(1-\theta_{10})+
  b\;(1-\theta_{11})}\:,
\end{equation}
and 
\begin{equation}\label{c1}
c_1={a\;\theta_{01}\over
  (1-a)\;\theta_{00}+ a\;\theta_{01}}={b\;\theta_{11}\over
  (1-b)\;\theta_{10}+ b\;\theta_{11}}\:.
\end{equation} 

Now consider condition 3 for $\alpha=\beta =1$. Symplifying using
\eref{c1}, we obtain:
\begin{equation}\label{eqcond3}
g_1(c_1) = Q_{11}= b \ \iff \ (1-a)\;\theta_{00}= b \;(1-\theta_{11})
\:.
\end{equation}
Condition 3 for other values of $\alpha$ and $\beta$ provide 
the same relation after simplification. 

Let us show that if equations \eref{d1} and \eref{eqcond3} are
satisfied, then the PCA also fulfills condition 1. Is is sufficient to
prove that $g_1(\pi_1)=\pi_1$. Expanding both sides of \eref{d2} and
simplifying using \eref{d1}, we obtain: 
\begin{equation}\label{fin}
a(1-b)\; (1-\theta_{01}-\theta_{10})=(1-a)b\; (1-\theta_{00}-\theta_{11})\: .
\end{equation}
Applying the definition \eref{eq-g2}, we have:
$$g_1(\pi_1)={1\over 1-b+a} \;
\Big((1-b)(1-a)\;\theta_{00}+(1-b)a\;\theta_{01}+a(1-b)\;\theta_{10}+ab\;\theta_{11}\Big)\: 
.$$ 
Using \eref{fin}, we can replace $a(1-b)(\theta_{01}+\theta_{10})$ by
$a(1-b)-(1-a)b\; (1-\theta_{00}-\theta_{11})$. With \eref{eqcond3}, we
finally obtain $g_1(\pi_1)=a/(1-b+a)=\pi_1$. 

Now, observe that the system:
\begin{equation}\label{system}
\left\{
\begin{array}{l}
(1-a) \; b \;\theta_{01}\;\theta_{10}=a\;(1-b)\;\theta_{00}\;\theta_{11}\\
(1-a)\;\theta_{00}= b \;(1-\theta_{11})
\end{array}\right.
\end{equation}
has a unique solution $(a,b)\in(0,1)^2$. Let $Q$ be the matrix
associated with $(a,b)$. Since the three above conditions are
satisfied, the Markovian measure $\nu_Q$ is invariant by the PCA. 
\end{proof}

In the Markovian case, unlike the Bernoulli case, 
there is no simple description of the law of other lines in the
stationary space-time diagram. 
Nevertheless, the stationary space-time diagram has a different
but still remakable property: it is {\em time-reversible}, meaning it
has the  same distribution if we
reverse the direction of time. This is proved in \cite{vas}.


\medskip

Bernoulli product measures are special cases of Markovian
measures. Therefore it is natural to ask whether all the cases covered
by Theorem \ref{iff} are retrieved in \eref{eq-mbm}. The
answer is no. Indeed, the measure $\nu_Q$ is a Bernoulli product
measure iff $a=b$. Simplifying in \eref{system} and \eref{eq-mbm},
we obtain:
\[
\bigl[ \theta_{00}=\theta_{01}, \theta_{11}=\theta_{10}\bigr] \quad
\textrm{or} \quad \bigl[ \theta_{00}=\theta_{10},
\theta_{11}=\theta_{01}\bigr] \:.
\]
The corresponding PCA have a neighborhood of size 1. This is far from
exhausting the PCA with a Bernoulli product measure. 

\paragraph{Finite set of cells.} It is also interesting to draw a
parallel between the result of Theorem \ref{theo:markov} and
Proposition 4.6 of Bousquet-M\'elou \cite{BoMe98}. In this last
article, the author studies PCA of alphabet $\A=\{0,1\}$ and
neighborhood $\Neighb=\{0,1\}$, but defined on a finite ring of size
$N$ (periodic boundary conditions: $X_N=X_0$), and proves that
the invariant measure has a Markovian form  if the
parameters satisfy the same relation \eref{eq-mbm} as in the infinite case.
The expression of the measure is then given by:
$$\P(X_0=x_0,X_1=x_1,\ldots,X_{N-1}=x_{N-1})={1\over
  Z}\prod_{i=0}^{N-1}q_{x_i,x_{i+1}},$$ 
where $Z$ is a normalizing constant, and where the coefficients $a$
and $b$ defining the matrix $Q$ are the solution of the same system
(\ref{system}) as in the infinite case. 

For a PCA satisfying condition \eref{eq-mbm}, we have a Markovian
invariant measure both on a
finite ring and on $\Z$. This is not the case for Bernoulli product measures: except when the
actual neighborhood is of size $1$, PCA satisfying the
conditions of Theorem \ref{iff} do not have a product form invariant
measure on finite rings. 

\begin{exam}\rm 
Consider for instance the PCA of transition
function $f(x,y)=(3/4) \; \delta_{x+y \mod 2}+(1/4)\;
\delta_{x+y+1 \mod 2}$ (Example \ref{ex:sum}), on the ring of size
4. Its invariant measure $\mu$ is different from the uniform measure:
\begin{eqnarray*}
\mu(0000) = 573/8192, \ \mu(0001)= 963/16384, \ \mu(0011) = 33/512, \\
\mu(0101) =  69/1024, \ \mu(0111)= 957/16384, \ \mu(1111) =
563/8192 \:.
\end{eqnarray*}
\end{exam}

\subsection{General alphabet and neighborhood}\label{larger}

In this section, the neighborhood is $\Neighb=\{0,\ldots, \ell \}$ and
the alphabet  is $\A=\{0,\ldots,n\}$. For
$p=(p_0,\ldots,p_n)$ such that $p_0+\ldots+p_n=1$, we still denote by
$\mu_p$ the corresponding Bernoulli product measure on $\A^{\Z}$.  

For convenience, we introduce the following notations: $\forall
x_0,\ldots, x_{\ell} \in \A, \ \forall k\in \A$, 
\begin{equation*}
\theta^{k}_{x_0\cdots x_{\ell}}=f(x_0,\ldots, x_{\ell})(k) \:.
\end{equation*}

We define new functions $g_{k}$ and $h_{k}$, which generalize the
ones in \eref{eq-g} and \eref{eq-h}. These new functions $g_{k}$ and $h_{k}$ are not functions of a
single variable, but of probability measures on
$\A^{\ell}$. Assume that:
\begin{equation}\label{eq-condition2}
\forall k\in \A, \ \forall x_0\cdots x_{\ell-1} \in \A^{\ell}, \
\exists i \in \A, \qquad \theta_{ x_0\cdots x_{\ell-1}i}^k >0 \:.
\end{equation}

Let us define:
\begin{eqnarray*}
g_{k}: \M(\A^{\ell}) & \longrightarrow & (0,1), \\
\D & \longmapsto & \hbox{ the probability that } Y_0=k \hbox{ if
} (X_0,\ldots,X_{\ell})\sim \D\otimes \ber_{p} \:,\\
& &\\
h_{k}: \M(\A^{\ell}) & \longrightarrow & \M(\A^{\ell}), \\
 \D & \longmapsto & \hbox{ the distribution of } (X_1,\ldots,X_{\ell}) \hbox{ conditionally to } Y_0=k \\
& & \hbox{ if } (X_0,\ldots,X_{\ell})\sim \D\otimes \ber_{p} \:.
\end{eqnarray*}

We have the following analog of Proposition \ref{formule_vois2}. 


\begin{prop}\label{formule0}  Consider a PCA satisfying
  \eref{eq-condition2}. Consider $p=(p_i)_{i\in \A}$ with $p_i>0$ for all $i$. For $\alpha_0\cdots \alpha_{n-1} \in \A^n$, the probability of the
cylinder $[\alpha_0\cdots \alpha_{n-1}]$ under $\mu_p F$ is given by:  
$$\mu_pF[\alpha_0\cdots\alpha_{n-1}]= g_{\alpha_0}(\ber_{p}^{\otimes \ell+1} ) \; \prod_{i=1}^{n-1}
g_{\alpha_i}(h_{\alpha_{i-1}}(h_{\alpha_{i-2}}(\ldots
h_{\alpha_0}(\ber_{p}^{\otimes \ell})\ldots )))\:.$$ 
\end{prop}

By reversing the space-direction, we get an analog of Proposition \ref{formule0} under the symmetric
condition: $\forall k\in \A,  \forall x_0\cdots x_{\ell-1} \in \A^{\ell}, 
\exists i \in \A, \ \theta_{ i x_0\cdots x_{\ell-1}}^k >0$. 

\medskip

Applying Proposition \ref{formule0}, we otain the following result. 
It already appears in \cite{vas} in a more complicated setting.

\begin{theo}\label{formule} Consider $p=(p_i)_{i\in \A}$ with $p_i>0$
  for all $i$.  The measure $\mu_p$ is an invariant measure of
  the PCA $F$  if one of
  the two following conditions is satisfied:  
\begin{eqnarray}
& \forall x_0,\ldots, x_{\ell-1}\in \A, \forall
k\in\A, &  \sum_{i\in\A} p_i\; \theta^k_{x_0 \cdots x_{\ell-1}i}=p_k  \label{eq-gencond1}\\
 & \forall x_0,\ldots, x_{\ell-1}\in \A, \forall
k\in\A, & \sum_{i\in\A}^n p_i\; \theta^k_{i x_0 \cdots
  x_{\ell-1}}=p_k \:.\label{eq-gencond2}
\end{eqnarray} 
\end{theo}

%
%
%

\begin{proof} Let us assume that $F$ satisfies condition  \eref{eq-gencond1}.
Then, the function $g_k$ is constant. Indeed,
$$g_{k}(\D)=\sum_{i\in\A, x_0\cdots x_{\ell-1}\in\A^{\ell}}\D(x_0,\ldots,x_{\ell-1})\;
p_i \; \theta^k_{x_0 \cdots  x_{\ell-1} i}=p_{k}\:.$$
By Proposition \ref{formule0}, we obtain that $\mu_p F=\mu_p$. 

Now, like in the proof of Theorem \ref{iff}, we can reverse the space
direction and define a new PCA $\widetilde{F}$. The PCA $F$ satisfies
condition  \eref{eq-gencond2} iff the PCA $\widetilde{F}$ satisfies 
condition \eref{eq-gencond1}. Therefore, if $F$ satisfies
condition \eref{eq-gencond2}, then we have $\mu_p\widetilde{F} =\mu_p$, which
implies in turn that $\mu_pF =\mu_p$. 
\end{proof}

As opposed to Theorem \ref{iff}, the conditions in Theorem
\ref{formule} are sufficient but not necessary. To illustrate this
fact, the  simplest examples are provided by PCA that do not depend on
all the elements of their neighborhood. 
Consider for instance the PCA of alphabet $\A=\{0,1\}$ and
neighborhood $\Neighb=\{0,1,2\}$, defined, for some $a,b\in (0,1)$, by:
$\forall u,v \in\A, \theta_{u0v}^1=a$, $\theta_{u1v}^1=b$. This PCA
has a Bernoulli invariant measure, but if $a\not=b$, it satisfies
neither condition \eref{eq-gencond1}, nor condition \eref{eq-gencond2}.

Let us state a result from \cite{vas}, which extends
Proposition \ref{pr-ergodic}, and completes Theorem \ref{formule}. 
(For the relevance of this result, see the discussion following Proposition
\ref{pr-ergodic}.)

\begin{prop}\label{pr-ergodic2}
Consider a positive-rates PCA $F$ satisfying condition
\eref{eq-gencond1} or  \eref{eq-gencond2}, for some $p=(p_i)_{i\in \A}$, $p_i>0$ for all $i$. 
Then $F$ is ergodic, that is, $\mu_p$ is the unique invariant
measure of $F$ and for all initial measure $\mu$, the sequence
$(\mu F^n)_{n\geq 0}$ converges weakly to $\mu_p$. 
\end{prop}



%
%
%

Condition \eref{eq-gencond1} implies that the variables $X_0,\ldots,
X_{\ell-1},Y_0$ are mutually independent, since for any
$v\in\{0,1\}^{\ell}$ and $\alpha\in\{0,1\}$, we have $\P((X_0,\ldots,
X_{\ell-1})=v,Y_0=\alpha)=\mu_p[v]\; \sum_{i\in\A}p_i\;
\theta_{vi}^{\alpha}=\mu_p[v]\mu_p[\alpha]$. Similarly, condition \eref{eq-gencond2}
implies that the variables $X_1,\ldots, X_{\ell},Y_0$ are
mutually independent.  

%
%
%
Next lemma is a generalization of Lemma \ref{pt_indep}.

\begin{lemm}\label{pt_indep_ger}
Under conditions \eref{eq-gencond1} and \eref{eq-gencond2}, the variable $X_0^0$ is independent of $(X_k^n)_{k\in\Z,n\in\N\setminus\{0\}}$.
\end{lemm}

\begin{proof} Set $X=X^0$ and $Y=X^1$. Like in Lemma \ref{pt_indep}, it is sufficient to prove
  that  $X_0$ is independent of $Y=(Y_k)_{k\in\Z}$. Let us fix some
  $a,b\in\Z, (a<0<b)$, and prove that $X_0$ is independent of
  $(Y_{a},Y_{a+1},\ldots,Y_b)$. We have: 
$$S=\P\bigl(X_0=x_0,(Y_i)_{a\leq i\leq b}=(y_i)_{a\leq i\leq b}\bigr)$$
$$=\sum_{\substack{x_i\in\A \\ i\in\{a,a+1,\ldots,b+\ell\}\setminus\{0\}}}  \P\bigl((X_i)_{a\leq i\leq b+\ell}=(x_i)_{a\leq i\leq b+\ell},(Y_i)_{a\leq i\leq b}=(y_i)_{a\leq i\leq b}\bigr)\:.$$
 Furthermore
 $$\P\bigl((X_i)_{a\leq i\leq b+\ell}=(x_i)_{a\leq i\leq b+\ell},(Y_i)_{a\leq i\leq b}=(y_i)_{a\leq i\leq b}\bigr)$$
$$\qquad =\mu_p[x_0] 
\prod_{i=a}^{-1}\mu_p[x_i]\; \theta_{x_i\cdots x_{i+\ell}}^{y_{i}}
\prod_{j=\ell}^{b+\ell}\mu_p[x_{j}]\; \theta_{x_{j-\ell}\cdots x_{j}}^{y_{j-\ell}}
\prod_{k=1}^{\ell-1}\mu_p[x_k] \:.
$$
If we compute the sum $S$ in the order: $x_a,\ldots,x_{-1}$ first
(simplifications using condition \eref{eq-gencond1}) then
$x_{b+\ell},x_{b+\ell-1},\ldots,x_{\ell}$ (simplifications using
condition \eref{eq-gencond2}), and finally $x_1,\ldots,x_{\ell-1}$, we obtain eventually:
$S=\mu_p[x_0]\, \prod_{i=a}^b\mu_p[y_i]$. 
\end{proof} 

\begin{cor} If both conditions \eref{eq-gencond1} and \eref{eq-gencond2} are
  satisfied, then every line of the stationary space-time diagram consists of
  i.i.d. random variables. In particular, any two different random variables are independent. 
\end{cor}

If the neighborhood is $\Neighb=\{0,1\}$, the spatial properties of
Section \ref{spatial} remain for a general alphabet (existence of
transversal PCA, properties of triangles,...). For other neighborhoods, there is no natural transversal PCA. 

\section{Cellular automata}

A {\em cellular automaton} (CA) is a PCA in which the transition
function $f$ is such that, for all $x\in \A^{\Neighb}$, the
probability measure $f(x)$ is concentrated on a single letter of the
alphabet. 
Thus, the transition function of a CA can be described by a
mapping $f:\A^{\Neighb}\longrightarrow \A$, and the CA can be viewed
as a deterministic mapping $F:\A^{\Z}\longrightarrow \A^{\Z}$. 

Cellular automata are classical and relevant mathematical objects:
they are precisely the mappings from $\A^{\Z}$ to $\A^{\Z}$ which are continuous (with respect to
the product topology) and commute with the shift, see
\cite{hedl}. 

\subsection{Known results}\label{sse-known}

\begin{defi}\label{de-perm}
A cellular automaton of transition function $f:\A^\Neighb \longrightarrow
\A$, where the neighborhood is of the form $\Neighb = \{\ell, \dots,
r-1,r\}$ for some $\ell<r$, is {\rm left-permutative} (resp. {\rm
  right-permutative}) if, for all
$w=w_{\ell}\cdots w_{r-1}\in \A^{r-\ell}$, the mapping from $\A$ to $\A$
defined by: $a \longmapsto f(aw)$ (resp. $a \longmapsto f(wa)$), is bijective.  A CA is {\rm
  permutative} if it is either left or
right-permutative. 
\end{defi}

Let $F:\A^{\Z}\longrightarrow \A^{\Z}$ be a permutative CA. The
existence of the bijections, see Definition \ref{de-perm}, has two direct
consequences:
$(i)$ $F$ is surjective; $(ii)$ the uniform measure is invariant: $\mu_{1/2}F=\mu_{1/2}$.
In fact, these last two properties are equivalent. 

\begin{prop}[Hedlund~\cite{hedl}]\label{pr-hedl}
Let $F$ be a cellular automaton. We have:
\[
F \mbox{ is surjective } \ \iff \ \mu_{1/2} F = \mu_{1/2} \:.
\]
\end{prop}

There exist surjective CA which are non permutative. 
Consider, for instance, the mapping $F_0:\{0,1\}^\Z\longrightarrow \{0,1\}^\Z$, defined
as follows. Set $A=10010$ and
$B=11000$. Observe that the two patterns $A$ and $B$ do not
overlap. From a
configuration $u\in \{0,1\}^\Z$, we get its image $F_0(u)$ by changing each
occurence of $A$ into $B$, resp. of $B$ into $A$. Clearly, the mapping $F_0$ can
be defined as a cellular automaton with neighborhood
$\Neighb=\{-4,,\dots, 0 , \dots, 4\}$.  Also, $F$ is surjective but
not permutative. 

\medskip

Let us present a recent result which refines
Proposition \ref{pr-hedl}. Given a finite and non-empty word $u\in \A^+$, let
$u^{\Z}=\cdots uuu\cdots  \in \A^{\Z}$ be a periodic bi-infinite word
of period $u$ (the starting position is indifferent). If $F:\A^\Z\longrightarrow \A^\Z$ is a CA, then
$F(u^\Z)=v^\Z$ for some word $v$ with $|v|=|u|$. For simplicity, we
write $v=F(u)$.

\begin{theo}[Kari-Taati~\cite{KaTa}]\label{th-kata}
Consider a CA $F$ on the alphabet $\A$. The Bernoulli product measure
$\mu_p$, $p=(p_i)_{i\in \A}$, $p_i>0$ for all $i$, is invariant for $F$ if and only if:
\[
(i) \ F \mbox{ is surjective } \qquad \hbox{ and } \qquad(ii) \ \forall u \in \A^+, \
\sum_{i \in \A} |u|_i \log(p_i) = \sum_{i\in \A} |F(u)|_i \log(p_i)
\:.
\]
\end{theo}

Let us mention two consequences of the above results. 

If a cellular automaton has an invariant Bernoulli product measure
$\mu_p$ ($p_i>0$ for all $i$), then the uniform measure is also
invariant. 

A cellular automaton $F$ is \emph{number-conserving} if: $\forall u
\in \A^+, \forall i \in \A, \ |u|_i = |F(u)|_i$. A surjective and
number-conserving CA admits 
all Bernoulli product measures
$\mu_p$ as invariant
measures. 
For instance, the CA $F_0$, defined above, is surjective and
number-conserving. Therefore, all the Bernoulli product measures are invariant
for $F_0$. 


\subsection{Link with the conditions for PCA}

The results in Sections \ref{cond}-\ref{spatial}-\ref{se:extensions}
give conditions for a PCA to admit 
invariant Bernoulli product measures. The above results, Section
\ref{sse-known}, give conditions for a CA to admit
invariant Bernoulli product measures.  The natural question is whether
we obtain the latter conditions by specializing the former ones. 

\medskip

Recall that the conditions \eref{eq-gencond1} or \eref{eq-gencond2} of
Theorem \ref{formule} are sufficient for the Bernoulli
product measure $\mu_p$ ($\forall i\in\A, p_i>0$) to be invariant for the
PCA $F$. 
Let us specialize these conditions to cellular automata, that is, let
us assume that all the coefficients $\theta^k_{x_0\ldots x_{\ell-1}i}$ are equal
to 0 or 1. 

\begin{lemm}
A cellular automaton satisfies condition \eref{eq-gencond1}, resp. \eref{eq-gencond2}, if and
only if it is right-permutative, resp. left-permutative. 
\end{lemm}

\begin{proof}
Consider a CA (transition function $f$) satisfying condition \eref{eq-gencond1} for some $p=(p_i)_{i\in\A}$. Set 
$J=\{j\in \A \mid p_j = \min_{i\in\A} p_i \}$ and consider $j\in J$. The
equality $p_j = \sum_{i\in\A} p_i\cdot \theta^j_{x_0\cdots
  x_{\ell-1}i}$, together with the constraints $\theta^j_{x_0\cdots
  x_{\ell-1}i}\in \{0,1\}$, implies that there must be exactly one
index $k\in J$ such that $\theta^j_{x_0\cdots
  x_{\ell-1}k} = 1$, i.e. $f(x_0,\ldots, x_{\ell -1},k)=j$. 
By repeating the argument, we obtain that for all
$x_0\cdots x_{\ell -1}$, the mapping $j\mapsto f(x_0,\ldots, x_{\ell
  -1},j)$ restricted to $J$ is a bijection. 
We now proceed by considering the set of indices $J_2= \{j\in \A - J
\mid p_j = \min_{i\in\A\setminus J} p_i \}$, and so on.
\end{proof}

To summarize, we recover the permutative CA. On the other hand, the
sujective but non-permutative CA are not captured by the sufficient
conditions of Theorem \ref{formule}. 

For a left-permutative CA (resp. right-permutative), the transversal CA,
see Section \ref{sse-conjug}, is right-permutative (resp. left-permutative), and
explicitly computable. Moreover, it is well-defined even if the space-time diagram is not assumed to be stationary.
We recover here a folk result.

In the special case $\A=\{0,1\}$ and $\Neighb=\{0,1\}$, all the
surjective CA are permutative. So in this case, we recover all the
surjective CA. This is consistent with the fact that in this case, 
the conditions of Theorem \ref{formule} are necessary and sufficient
(see Theorem \ref{iff}). 

\medskip

{\bf Remark.} Condition \eref{eq-gencond1} can be
interpreted as ``\emph{being right-permutative in expectation}'' for a
PCA. And similarly, condition \eref{eq-gencond2} amounts to ``\emph{being
  left-permutative in expectation}''. 

\section{Related open issues}

Consider a PCA of alphabet and neighborhood of size 2. 
Under the relations \eref{eq-bernou} or \eref{eq-mbm}, it has an
explicit invariant measure
with a simple form  (Bernoulli product
or Markovian). The conditions \eref{eq-bernou} and \eref{eq-mbm} are of codimension 1
in the parameter space. What happens for other values of the
parameters? Is it still possible to give an explicit description of
the invariant measure? This is an open
question. It has been deeply investigated for the family of PCA's
defined by: $\theta_{00}=\theta_{01}=\theta_{10}=a,
\theta_{11}=1-a$, for some $a\in (0,1)$. 
Observe that neither \eref{eq-bernou} nor \eref{eq-mbm} is satisfied except in the trivial case
$a=1/2$. The specific interest for these PCA is due to a connection with
directed animals and
percolation theory first noticed by Dhar~\cite{dhar83}, see also \cite{BoMe98,LeMa}. 
More specifically, determining explicitly the invariant measure for
the above PCA would enable to: (1) compute the area and perimeter
generating function of directed animals in the square lattice; (2)
compute the directed site-percolation threshold in the square
lattice. The most recent efforts to compute the invariant measure can be found in
\cite{marck12}.

%



\end{document}